\documentclass[12pt,a4paper]{article}
\usepackage{oldgerm}
\usepackage{amssymb}
\usepackage{amsthm}
\newtheorem{thm}{Theorem}[section]
\newtheorem{prop}[thm]{Proposition}
\newtheorem{lem}[thm]{Lemma}
\newtheorem{definition}[thm]{Definition}

\begin{document}
\title{The holomorphic symplectic structures on hyper-K\"ahler manifolds of type $A_{\infty}$}
\author{Kota Hattori}
\date{}
\maketitle
{\begin{center}
{\it Graduate School of Mathematical Sciences, University of Tokyo\\
3-8-1 Komaba, Meguro, Tokyo 153-8914, Japan\\
khattori@ms.u-tokyo.ac.jp}
\end{center}}
\maketitle
{\abstract Hyper-K\"ahler manifolds of type $A_{\infty}$ are noncompact complete Ricci-flat K\"ahler manifolds of complex dimension $2$, constructed by Anderson, Kronheimer, LeBrun \cite{AKL} and Goto \cite{G1}. We study the holomorphic symplectic structures preserved by the natural $\mathbb{C}^{\times}$-actions on these manifolds, then show the sufficient and necessary conditions for the existence of $\mathbb{C}^{\times}$-equivariant biholomorphisms between two hyper-K\"ahler manifolds of type $A_{\infty}$ preserving their holomorphic symplectic structures. As a consequence, we can show the existence of a complex manifold of dimension $2$ on which there is a continuous family of complete Ricci-flat K\"ahler metrics with distinct volume growth.}
\section{Introduction}
Hyper-K\"ahler manifolds of type $A_{\infty}$ were first constructed by Anderson, Kronheimer and LeBrun in \cite{AKL}, as the first example of complete Ricci-flat K\"ahler manifolds with infinite topological type. Here, infinite topological type means that their homology groups are infinitely generated. After \cite{AKL}, Goto \cite{G1,G2} has succeeded in constructing these manifolds in another way, using hyper-K\"ahler quotient construction. He also constructed the higher dimensional complete hyper-K\"ahler manifolds with infinite topological type. Some of the topological and geometric properties of hyper-K\"ahler manifolds of type $A_{\infty}$ were studied well in the above papers, and the author studied the volume growth of the Riemannian metrics in \cite{Hat}. Then this paper focuses on the complex geometry on the hyper-K\"ahler manifolds of type $A_{\infty}$.

A hyper-K\"ahler manifold is, by definition, a Riemannian manifold $(X,g)$ of real dimension $4n$ equipped with three complex structures $I_1,I_2,I_3$ satisfying the quaternionic relations $I_1^2 = I_2^2 = I_3^2 = I_1I_2I_3 = -{\rm id}$ with respect to all of which the metric $g$ is K\"ahlerian. Then the holonomy group of $g$ is a subgroup of $Sp(n)$ and $g$ is Ricci-flat. By the definition, the hyper-K\"ahler manifold carry three K\"ahler forms defined by $\omega_i:=g(I_i\cdot,\cdot)$ for $i=1,2,3$, then we have an non-degenerate closed $(2,0)$-form $\omega_{\mathbb{C}} = \omega_3 - \sqrt{-1}\omega_2$, which is called a holomorphic symplectic form, if $(X, I_1)$ is regarded as a complex manifold. Conversely, it is known that $(g,I_1,I_2,I_3)$ are reconstructed from $(\omega_1,\omega_2,\omega_3)$, thus we call $\omega=(\omega_1,\omega_2,\omega_3)$ the hyper-K\"ahler structure over $X$. In this paper we regard the hyper-K\"ahler manifold $(X,\omega)$ as a complex manifold with respect to the complex structure $I_1$.

In \cite{Hat}, the author computed the volume growth of hyper-K\"ahler manifolds of type $A_{\infty}$. Here, the volume growth of a Riemannian manifold $(X,g)$ is the asymptotic behavior of the function $V_g(p_0,r)$, which is defined as the volume of the geodesic ball of radius $r$ centered at $p_0\in X$. Then the following result was obtained.
\begin{thm}[\cite{Hat}]
There exists a $C^{\infty}$ manifold $X$ of ${\rm dim}_{\mathbb{R}} X = 4$ and a family of hyper-K\"ahler structures $\omega^{(\alpha)}$ on $X$ for $3 < \alpha < 4$ which carry complete hyper-K\"ahler metrics $g^{(\alpha)}$ with
\begin{eqnarray}
0 < \liminf_{r\to +\infty}\frac{V_{g^{(\alpha)}}(p_0,r)}{r^{\alpha}} \le \limsup_{r\to +\infty}\frac{V_{g^{(\alpha)}}(p_0,r)}{r^{\alpha}} < +\infty.\nonumber
\end{eqnarray}
for any $p_0\in X$.
\label{previous result}
\end{thm}

The hyper-K\"ahler manifolds $(X,\omega^{(\alpha)})$ in the above theorem are already constructed in \cite{AKL}\cite{G1}, and the essential part of Theorem \ref{previous result} is the computation of the volume growth of these manifolds.
In this paper we study the holomorphic symplectic structure $\omega^{(\alpha)}_{\mathbb{C}}$ over $X$. The period of $(X,\omega^{(\alpha)}_{\mathbb{C}})$, that is the cohomology class determined by $\omega^{(\alpha)}_{\mathbb{C}}$, is independent of $\alpha$. Then the holomorphic symplectic structures $\omega^{(\alpha)}_{\mathbb{C}}$ are expected to be independent of $\alpha$. We actually obtain the following result.
\begin{thm}
There exists a complex manifold $X$ of ${\rm dim}_{\mathbb{C}} X = 2$ and a family of complete Ricci-flat K\"ahler metrics $g^{(\alpha)}$ on $X$ for $3<\alpha <4$ with
\begin{eqnarray}
0 < \liminf_{r\to +\infty}\frac{V_{g^{(\alpha)}}(p_0,r)}{r^{\alpha}} \le \limsup_{r\to +\infty}\frac{V_{g^{(\alpha)}}(p_0,r)}{r^{\alpha}} < +\infty.\nonumber
\end{eqnarray}
for any $p_0\in X$.
\label{main theorem}
\end{thm}
Theorem \ref{main theorem} will be proved as a corollary of the results in \cite{Hat} and undermentioned Theorem \ref{main theorem2}, which is the main result in this paper.

It is known that there are complete Ricci-flat K\"ahler metrics over $\mathbb{C}^n$ who do not have the Euclidean volume growth \cite{L}. For example, the Taub-NUT metrics over $\mathbb{C}^2$ are the complete Ricci-flat K\"ahler metrics whose volume growth are $r^3$. 
On the other hand, Theorem \ref{main theorem} asserts the existence of a complex manifold who has a continuous family of complete Ricci-flat K\"ahler metrics, whose volume growth also change continuously.

To show that two hyper-K\"ahler quotients are biholomorphic or not, it is useful to see the GIT quotient construction and study the period.
For example,
Konno \cite{Ko2} has studied the period of holomorphic symplectic structures of toric hyper-K\"ahler manifolds, that are typical examples of hyper-K\"ahler quotients, using GIT quotient construction.
However, this method is not enough for studying the case of hyper-K\"ahler manifolds of type $A_{\infty}$,
because these manifolds are obtained by taking quotients by the action of infinite dimensional Lie groups on infinite dimensional manifolds,
then we should develop other methods to show that $(X,\omega^{(\alpha_1)}_{\mathbb{C}})$ and $(X,\omega^{(\alpha_2)}_{\mathbb{C}})$ are biholomorphic.

In this paper, we consider when two hyper-K\"ahler manifolds of type $A_{\infty}$ become isomorphic as holomorphic symplectic manifolds. Let $(X_i,\omega_{i})$ be the hyper-K\"ahler manifolds of type $A_{\infty}$ for $i=0,1$. Then there are the natural $\mathbb{C}^{\times}$-actions over $X_i$ preserving their holomorphic symplectic structures $\omega_{i,\mathbb{C}}$, and the complex moment maps $\mu_{i,\mathbb{C}}:X_i\to\mathbb{C}$. Since the complex moment maps are $\mathbb{C}^{\times}$-invariant, they define complex valued continuous functions $[\mu_{i,\mathbb{C}}]:X_i/\mathbb{C}^{\times} \to \mathbb{C}$ on the quotient topological spaces $X_i/\mathbb{C}^{\times}$. Moreover, the $\mathbb{C}^{\times}$-actions define natural partial order structures on the quotient spaces $X_i/\mathbb{C}^{\times}$. Then we obtain the following result.
\begin{thm}
There exists a $\mathbb{C}^{\times}$-equivariant biholomorphic map $f:X_0\to X_1$ with $f^*\omega_{1,\mathbb{C}} = \omega_{0,\mathbb{C}}$ if and only if there is a homeomorphism $\mathbf{h} : X_0/\mathbb{C}^{\times} \to X_1/\mathbb{C}^{\times}$ preserving the order structures and $[\mu_{1,\mathbb{C}}]\circ\mathbf{h} - [\mu_{0,\mathbb{C}}] $ is constant.
\label{main theorem2}
\end{thm}

The above theorem is proven as follows in this paper. Put
\begin{eqnarray}
X_i^*:=X_i\backslash \{ p\in X_i;\ pg=p \ {\rm for\ all\ }g\in\mathbb{C}^{\times}.\}.\nonumber
\end{eqnarray}
Then we have an open covering $X_i^*=\bigcup_s X_i^s$ where each $X_i^s$ is biholomorphic to $\mathbb{C}^{\times}\times\mathbb{C}$, consequently biholomorphic maps $X_0^s \to X_1^s$ are obtained. Moreover we can show that these biholomorphic maps glue on the intersections, therefore a biholomorphic map $X_0^* \to X_1^*$ is obtained, and it extends to the biholomorphic map $X_0 \to X_1$ by Hartogs' extension theorem. The similar way is used already in \cite{L} to show that the complex structure over $\mathbb{R}^4$ given by the Taub-NUT metric is biholomorphic to $\mathbb{C}^2$. But in our case, the topological structure of $X_i/\mathbb{C}^{\times}$ is so complicated that we should study them precisely.

This paper is organized as follows. 
First of all we review the results obtained in \cite{G1} in Sections 2 and 3. In Section 2, we construct hyper-K\"ahler manifolds of type $A_{\infty}$ by using hyper-K\"ahler quotient constructions and see that there exists a natural $S^1$-action.
In Section 3, we review the another quotient construction of hyper-K\"ahler manifolds of type $A_{\infty}$, and see that the manifolds obtained in Sections 2 and 3 are isomorphic as the holomorphic symplectic manifolds.

In Section 4, we study the topological properties of topological quotient spaces obtained from hyper-K\"ahler manifolds of type $A_{\infty}$ by taking the quotient by $\mathbb{C}^{\times}$-action. We can also see that there exist natural partial order structures.

In Section 5 we construct biholomorphisms between two hyper-K\"ahler manifolds of type $A_{\infty}$, which satisfy the assumption of Theorem \ref{main theorem2}. As a consequence, we apply Theorem \ref{main theorem2} for more concrete case, and obtain Theorem \ref{main theorem} and other results in Section 6.

\section{Hyper-K\"ahler manifolds of type $A_{\infty}$}
\subsection{Hyper-K\"ahler quotient construction}
\quad In this section, we review shortly the construction of hyper-K\"ahler manifolds of type $A_{\infty}$ along \cite{G1}. Although they can be constructed by Gibbons-Hawking ansatz \cite{AKL}, we need hyper-K\"ahler quotient construction in \cite{G1} for arguments in Section 4. For more details of construction and basic facts, see \cite{AKL}\cite{G1} or review in Section 2 of \cite{Hat}.

First of all, we describe the definition of hyper-K\"ahler manifolds.
\begin{definition}
{\rm Let $(X,g)$ be a Riemannian manifold of dimension $4n$, $I_1,I_2,I_3$ be integrable complex structures on $X$, and $g$ is a hermitian metric with respect to each $I_i$. Then $(X,g,I_1,I_2,I_3)$ is a hyper-K\"ahler manifold if $(I_1,I_2,I_3)$ satisfies the relations $I_1^2=I_2^2=I_3^2=I_1I_2I_3=-1$ and each fundamental $2$-form $\omega_i:=g(I_i\cdot,\cdot)$, that is, $(X,g,I_i)$ is k\"ahlerian.}
\end{definition}
Let $\mathbb{H}=\mathbb{R}\oplus\mathbb{R}i\oplus\mathbb{R}j\oplus\mathbb{R}k=\mathbb{C}\oplus\mathbb{C}j$ be quaternion and ${\rm Im}\mathbb{H} = \mathbb{R}i\oplus\mathbb{R}j\oplus\mathbb{R}k$ be its Imaginary part. Then an ${\rm Im}\mathbb{H}$-valued $2$-form $\omega:=i\omega_1 + j\omega_2 + k\omega_3\in\Omega^2(X)\otimes {\rm Im}\mathbb{H}$ are constructed from the hyper-K\"ahler structure $(g,I_1,I_2,I_3)$. Conversely, $(g,I_1,I_2,I_3)$ is reconstructed from $\omega$. Hence we call $\omega$ the hyper-K\"ahler structure on $X$ instead of $(g,I_1,I_2,I_3)$.\\

To construct hyper-K\"ahler manifolds of type $A_{\infty}$, we prepare an infinite countable set $\mathbf{I}$ and a parameter space
\begin{eqnarray}
({\rm Im}\mathbb{H})_0^{\mathbf{I}} := \{\lambda = (\lambda_n)_{n\in\mathbf{I}}\in({\rm Im}\mathbb{H})^{\mathbf{I}};\ \sum_{n\in \mathbf{I}}\frac{1}{1 + |\lambda_n|}<+\infty\}.\nonumber
\end{eqnarray}
For a set $S$, we denote by $S^{\mathbf{I}}$ the set of all maps from $\mathbf{I}$ to $S$. An element of $x\in S^{\mathbf{I}}$ is written as $x=(x_n)_{n\in\mathbf{I}}$. Then we have a Hilbert space
\begin{eqnarray}
M_{\mathbf{I}}:=\{v\in\mathbb{H}^{\mathbf{I}};\ \| v\|_{\mathbf{I}}^2<+\infty\},\nonumber
\end{eqnarray}
where
\begin{eqnarray}
\langle u,v\rangle_{\mathbf{I}}:=\sum_{n\in \mathbf{I}}u_n\bar{v}_n,\quad \| v\|^2_{\mathbf{I}}:=\langle v,v\rangle_{\mathbf{I}}\nonumber
\end{eqnarray}
for $u,v\in \mathbb{H}^{\mathbf{I}}$. Here, $\bar{v}_n\in\mathbb{H}$ is the quaternionic conjugate of $v_n$ defined by $\overline{a + bi +cj + dk} := a - bi - cj - dk$ for $a,b,c,d\in\mathbb{R}$.

Now we fix $\lambda\in ({\rm Im}\mathbb{H})_0^{\mathbf{I}}$, and take $\Lambda\in\mathbb{H}^{\mathbf{I}}$ to be $\Lambda_n i\overline{\Lambda}_n = \lambda_n$.
Then we have the following Hilbert manifolds
\begin{eqnarray}
M_{\Lambda} &:=& \Lambda + M_{\mathbf{I}} = \{\Lambda + v ;\ v\in M_{\mathbf{I}}\},\nonumber\\
G_{\lambda} &:=& \{g\in (S^1)^{\mathbf{I}};\ \sum_{n\in\mathbf{I}}(1 + |\lambda_n|) |1-g_n |^2 <+\infty,\ \prod_{n\in\mathbf{I}}g_n=1\},\nonumber\\
\mathbf{g}_{\lambda} &:=& {\rm Lie}(G_{\lambda}) = \{\xi\in \mathbb{R}^{\mathbf{I}};\ \sum_{n\in\mathbf{I}}(1 + |\lambda_n|) |\xi_n |^2<+\infty, \ \sum_{n\in\mathbf{I}}\xi_n=0\}.\nonumber
\end{eqnarray}
The convergence of $\prod_{n\in\mathbf{I}}g_n$ and $\sum_{n\in\mathbf{I}}\xi_n$ follows from the condition $\sum_{n\in \mathbf{I}}(1+|\lambda_n|)^{-1}<+\infty$.
Then $G_{\lambda}$ is a Hilbert Lie group whose Lie algebra is $\mathbf{g}_{\lambda}$. We can define a right action of $G_{\lambda}$ on $M_{\Lambda}$ by $xg:=(x_ng_n)_{n\in \mathbf{I}}$ for $x\in M_{\Lambda}$, $g\in G_{\lambda}$. 
Here the product of $x_n$ and $g_n$ is given by regarding $S^1$ as the subset of $\mathbb{H}$ by the natural injections $S^1\subset\mathbb{C}\subset\mathbb{H}$.
Then $G_{\lambda}$ acts on $M_{\Lambda}$ preserving the hyper-K\"ahler structure, and we have the hyper-K\"ahler moment map $\hat{\mu}_{\Lambda}:M_{\Lambda}\rightarrow {\rm Im}\mathbb{H}\otimes\mathbf{g}_{\lambda}^*$ defined by
\begin{eqnarray}
\langle\hat{\mu}_{\Lambda}(x),\xi\rangle :=\sum_{n\in \mathbf{I}}(x_ni\bar{x}_n - \Lambda_ni\overline{\Lambda}_n)\xi_n\ \in {\rm Im}\mathbb{H}\nonumber
\end{eqnarray}
for $x \in M_{\Lambda}$, $\xi\in\mathbf{g}_{\lambda}$. If $\mathbf{I}$ is a finite set.

Since $\hat{\mu}_{\Lambda}$ is $G_{\lambda}$-invariant, then $G_{\lambda}$ acts on the inverse image
\begin{eqnarray}
\hat{\mu}_{\Lambda}^{-1}(0) = \{ x\in M_{\Lambda};\ x_ni\overline{x}_n - \lambda_n = x_mi\overline{x}_m - \lambda_m \ {\rm for\ all\ }n,m\in \mathbf{I}\}\nonumber
\end{eqnarray}
Hence we obtain the quotient space $\hat{\mu}_{\Lambda}^{-1}(0)/G_{\lambda}$ which is called the hyper-K\"ahler quotient.
\begin{definition}
{\rm An element $\lambda\in ({\rm Im}\mathbb{H})_0^{\mathbf{I}}$ is generic if $\lambda_n - \lambda_m \neq 0$ for all distinct $n,m\in\mathbf{I}$.}
\end{definition}
\begin{thm}[\cite{G1}]
If $\lambda\in ({\rm Im}\mathbb{H})_0^{\mathbf{I}}$ is generic, then $\hat{\mu}_{\Lambda}^{-1}(0)/G_{\lambda}$ is a smooth manifold of real dimension $4$, and the hyper-K\"ahler structure on $M_{\Lambda}$ induces a hyper-K\"ahler structure $\omega_{\lambda}$ on $\hat{\mu}_{\Lambda}^{-1}(0)/G_{\lambda}$.
\label{2.5}
\end{thm}
The quotient space $\hat{\mu}_{\Lambda}^{-1}(0)/G_{\lambda}$ seems to depend on the choice of $\Lambda\in \mathbb{H}^\mathbf{I}$, but the induced hyper-K\"ahler structure on $\hat{\mu}_{\Lambda}^{-1}(0)/G_{\lambda}$ depends only on $\lambda$ from the argument of Section 2 of \cite{Hat}.
Thus we may put
\begin{eqnarray}
X_{HKQ}(\lambda) &:=& \hat{\mu}_{\Lambda}^{-1}(0)/G_{\lambda}\nonumber\\
&=& \{x\in M_{\Lambda}; x_n i\bar{x}_n - \lambda_n\ {\rm is\ independent\ of}\ n\in\mathbf{I}\}/G_{\lambda},\nonumber
\end{eqnarray}
and call it hyper-K\"ahler manifold of type $A_{\infty}$

Recall that we assume that $\mathbf{I}$ is infinite. If $\sharp\mathbf{I}=k + 1 < +\infty$, then $(X_{HKQ}(\lambda),\omega_{\lambda})$ becomes an ALE hyper-K\"ahler manifold of type $A_k$ \cite{G3}.

\subsection{$S^1$-actions and moment maps\label{GH ansatz}}\label{sec2.2}
In \cite{G1}, an $S^1$-action on $X_{HKQ}(\lambda)$ preserving the hyper-K\"ahler structure defined as follows.
We denote by $[x]\in \hat{\mu}_{\Lambda}^{-1}(0)/G_{\lambda}$ the equivalence class represented by $x\in\hat{\mu}_{\Lambda}^{-1}(0)$. Fix $m\in \mathbf{I}$ and put
\begin{eqnarray}
[x]g:=[x_mg,(x_n)_{n\in \mathbf{I}\backslash\{m\}}]\nonumber
\end{eqnarray}
for $x=(x_m,(x_n)_{n\in \mathbf{I}\backslash\{m\}})\in\hat{\mu}_{\Lambda}^{-1}(0)$ and $g\in S^1$. This definition is independent of the choice of $m\in \mathbf{I}$, and we have the action of $S^1$ on $X_{HKQ}(\lambda)$. The hyper-K\"ahler moment map $\mu_{\lambda}:X_{HKQ}(\lambda)\rightarrow {\rm Im}\mathbb{H}=\mathbb{R}^3$ is defined by
\begin{eqnarray}
\mu_{\lambda}([x]) := x_n i\bar{x}_n - \lambda_n \in {\rm Im}\mathbb{H}.\nonumber
\end{eqnarray}
The right hand side is independent of the choice of $n\in\mathbf{I}$ since $x$ is an element of $\hat{\mu}_{\Lambda}^{-1}(0)$.

Put
\begin{eqnarray}
X_{HKQ}(\lambda)^* &:=& \{[x]\in X_{HKQ}(\lambda);\ x_n\neq 0\ {\rm for\ all\ }n\in\mathbf{I}\},\nonumber\\
Y_{\lambda} &:=& {\rm Im}\mathbb{H}\backslash \{- \lambda_n;\ n\in\mathbf{I}\}, \nonumber
\end{eqnarray}
then we have a principal $S^1$-bundle $\mu_{\lambda}\big|_{X_{HKQ}(\lambda)^*}:X_{HKQ}(\lambda)^*\rightarrow Y_{\lambda}$, and $S^1$ acts on $X_{HKQ}(\lambda)\backslash X_{HKQ}(\lambda)^*$ trivially.

Conversely, on the total spaces of some principal $S^1$-bundle over $Y_{\lambda}$, hyper-K\"ahler structures preserved by $S^1$-actions are constructed in \cite{AKL} by Gibbons-Hawking ansatz.
It is shown in \cite{G1} that each $X_{HKQ}(\lambda)$ is isomorphic to one of that constructed by Gibbons-Hawking ansatz.

By observing the Gibbons-Hawking construction, it is easy to see that $X_{HKQ}(\lambda)$ and $X_{HKQ}(\lambda')$ are isomorphic as hyper-K\"ahler manifolds if $\lambda$ and $\lambda'$ satisfy one of the following relations; $(i)$ $\lambda'_n - \lambda_n \in{\rm Im}\mathbb{H}$ are independent of $n$, $(ii)$ $\lambda'_n = \lambda_{a(n)}$ for some bijective maps $a: {\rm Im}\mathbb{H} \to {\rm Im}\mathbb{H}$.

We can also show easily that $X_{HKQ}(\lambda) \cong X_{HKQ}(- \lambda)$ by constructing an isomorphism explicitly.

\section{Holomorphic description}

In this section we compare hyper-K\"ahler quotients $\hat{\mu}_{\Lambda}^{-1}(0)/G_{\lambda}$ with another kind of quotient spaces $\hat{\mu}_{\Lambda,\mathbb{C}}^{-1}(0)/G_{\lambda}^{\mathbb{C}}$ along \cite{G1}, where $\hat{\mu}_{\Lambda,\mathbb{C}}$ is the complex valued component of $\hat{\mu}_{\Lambda}$ and $G_{\lambda}^{\mathbb{C}}$ is the complexification of $G_{\lambda}^{\mathbb{C}}$.

First of all, we complexify the Hilbert Lie group $G_{\lambda}$ as follows,
\begin{eqnarray}
G_{\lambda}^{\mathbb{C}} &:=& \{g\in (\mathbb{C}^{\times})^{\mathbf{I}};\ \sum_{n \in \mathbf{I}} (1 + |\lambda_n|) |1 - g_n |^2<+\infty, \ \prod_{n\in\mathbf{I}}g_n=1\},\nonumber\\
\mathbf{g}_{\lambda}^{\mathbb{C}} &:=& \mathbf{g}_{\lambda}\otimes\mathbb{C} = \{\xi\in \mathbb{C}^{\mathbf{I}};\ \sum_{n \in \mathbf{I}} (1 + |\lambda_n|) |\xi_n |^2<+\infty, \ \sum_{n\in\mathbf{I}}\xi_n=0\},\nonumber
\end{eqnarray}
where $\mathbb{C}^{\times} = \mathbb{C}\backslash\{ 0 \}$. Then $G_{\lambda}^{\mathbb{C}}$ acts smoothly on $M_{\Lambda}$, where $\Lambda\in\mathbb{H}^\mathbf{I}$ satisfies $\Lambda_n i \overline{\Lambda}_n = \lambda_n$.

From now on we write $\zeta = \zeta_{\mathbb{R}}i - \zeta_{\mathbb{C}}k = (\zeta_{\mathbb{R}}, \zeta_{\mathbb{C}})\in {\rm Im}\mathbb{H}$ along the decomposition ${\rm Im}\mathbb{H} = \mathbb{R}i \oplus \mathbb{C}k$. Similarly, we write $\lambda = \lambda_{\mathbb{R}}i - \lambda_{\mathbb{C}}k = (\lambda_{\mathbb{R}}, \lambda_{\mathbb{C}})$ for $\lambda\in ({\rm Im}\mathbb{H})^\mathbf{I}$, where $\lambda_{\mathbb{R}}\in\mathbb{R}^\mathbf{I}$ and $\lambda_{\mathbb{C}}\in\mathbb{C}^\mathbf{I}$.
The hyper-K\"ahler moment map $\hat{\mu}_{\Lambda}$ is also decomposed into two components as $\hat{\mu}_{\Lambda} = \hat{\mu}_{\Lambda,\mathbb{R}}\cdot i - \hat{\mu}_{\Lambda,\mathbb{C}}\cdot k$. 
Then $\hat{\mu}_{\Lambda,\mathbb{R}}:M_{\Lambda}\to\mathbf{g}_{\lambda}^*$ and $\hat{\mu}_{\Lambda,\mathbb{C}}:M_{\Lambda}\to(\mathbf{g}_{\lambda}^{\mathbb{C}})^*$ are written as
\begin{eqnarray}
\langle \hat{\mu}_{\Lambda,\mathbb{R}}(z + wj), \xi \rangle &=& \sum_{n\in\mathbf{I}}(|z_n|^2 - |w_n|^2 - \lambda_{n,\mathbb{R}})\xi_n,\nonumber\\
\langle \hat{\mu}_{\Lambda,\mathbb{C}}(z + wj), \eta \rangle &=& \sum_{n\in\mathbf{I}}(2z_n w_n - \lambda_{n,\mathbb{C}})\eta_n,\nonumber
\end{eqnarray}
for $z+wj = (z_n + w_n j)_{n\in\mathbf{I}} \in M_{\Lambda}$, $\xi\in\mathbf{g}_{\lambda}$ and $\eta\in\mathbf{g}_{\lambda}^{\mathbb{C}}$, where $z_n,w_n\in\mathbb{C}$ and $\lambda_n = \lambda_{n,\mathbb{R}} i - \lambda_{n,\mathbb{C}}k$. Then $\hat{\mu}_{\Lambda,\mathbb{C}}$ is $G_{\lambda}^{\mathbb{C}}$ invariant.
\begin{definition}
{\rm Let $t=(t_n)_{n\in\mathbf{I}}\in\mathbb{R}^{\mathbf{I}}$. Then $z+wj \in M_{\Lambda}$ is $t$-stable if $|z_n|^2 + |w_m|^2 > 0$ holds for any $n,m\in\mathbf{I}$ which satisfy $t_n > t_m$.}
\end{definition}
Now we put
\begin{eqnarray}
\hat{\mu}_{\Lambda,\mathbb{C}}^{-1}(0)_t := \{z + wj\in \hat{\mu}_{\Lambda,\mathbb{C}}^{-1}(0); z + wj\ {\rm is}\ t {\rm \mathchar` stable}\}.\nonumber
\end{eqnarray}
Then $G_{\lambda}^{\mathbb{C}}$ acts on $\hat{\mu}_{\Lambda,\mathbb{C}}^{-1}(0)_t$. If the quotient space 
$X_{GIT}(\lambda) := \hat{\mu}_{\Lambda,\mathbb{C}}^{-1}(0)_{\lambda_{\mathbb{R}}} /G_{\lambda}^{\mathbb{C}}$
becomes a smooth manifold, then the standard nowhere vanishing $(2,0)$-form $\sum_{n\in\mathbf{I}}dz_n\wedge dw_n$ over $M_{\Lambda}$ induces a holomorphic symplectic form $\omega_{\lambda,\mathbb{C}}$ on $X_{GIT}(\lambda)$. Then $(X_{GIT}(\lambda), \omega_{\lambda,\mathbb{C}})$ depends only on $\lambda$, not depends on $\Lambda$.

\begin{thm}[\cite{G1}]
Let $\lambda\in ({\rm Im}\mathbb{H})_0^\mathbf{I}$ be generic. Then the quotient space $X_{GIT}(\lambda)$ becomes a complex manifold of dimension $2$.
\end{thm}

For any generic $\lambda$, $\hat{\mu}_{\Lambda}^{-1}(0)$ is a subset of  $\hat{\mu}_{\Lambda,\mathbb{C}}^{-1}(0)_{\lambda_{\mathbb{R}}}$. Then this inclusion induces
\begin{eqnarray}
\phi_{\lambda}:X_{HKQ}(\lambda) = \hat{\mu}_{\Lambda}^{-1}(0) /G_{\lambda}  \to  \hat{\mu}_{\Lambda,\mathbb{C}}^{-1}(0)_{\lambda_{\mathbb{R}}} /G_{\lambda}^{\mathbb{C}} = X_{GIT}(\lambda),\nonumber
\end{eqnarray}
which is an biholomorphism preserving the holomorphic structure, namely,
\begin{eqnarray}
\phi_{\lambda}^*\omega_{\lambda,\mathbb{C}} = \omega_{\lambda,3} - \sqrt{-1}\omega_{\lambda,2},\nonumber
\end{eqnarray}
where $\omega_{\lambda}=\omega_{\lambda,1}i + \omega_{\lambda,2}j + \omega_{\lambda,3}k$ is the hyper-K\"ahler structure on $X_{HKQ}(\lambda)$. Here, $\omega_{\lambda,3} - \sqrt{-1}\omega_{\lambda,2}$ is the holomorphic symplectic structures over $X_{HKQ}(\lambda)$ with respect to the complex structure $I_{\lambda,1}$.
From now on we write
\begin{eqnarray}
(X(\lambda),\omega_{\lambda,\mathbb{C}}) := (X_{HKQ}(\lambda),\omega_{\lambda,3} - \sqrt{-1}\omega_{\lambda,2}) = (X_{GIT}(\lambda),\omega_{\lambda,\mathbb{C}})\nonumber
\end{eqnarray}
if it is not necessary to distinguish them.

In Section \ref{sec2.2}, we have seen that $X(\lambda)$ has a natural $S^1$-action.
Then by complexifying the action, we have a holomorphic $\mathbb{C}^{\times}$-action on $X(\lambda)$ preserving $\omega_{\lambda,\mathbb{C}}$ defined by
\begin{eqnarray}
[z+w j]g:=[z_m g + w_m g^{-1},(z_n + w_n)_{n\in \mathbf{I}\backslash\{m\}}].\nonumber
\end{eqnarray}
It is easy to see that $\mathbb{C}^{\times}$ acts freely on $X(\lambda)^*=X_{HKQ}(\lambda)^*$, and trivially on $X(\lambda)\backslash X(\lambda)^*$.

\section{Topological structure of $X(\lambda)/\mathbb{C}^{\times}$}

In the previous section, we obtain $\mathbb{C}^{\times}$-action on $X(\lambda)$. 
In this section we will study the topology of the quotient space $X(\lambda)/\mathbb{C}^{\times}$ with the quotient topology.

\subsection{The topological space homeomorphic to $X(\lambda)/\mathbb{C}^{\times}$}

First of all, we define a certain equivalence relation $\sim_{\lambda}$ in ${\rm Im}\mathbb{H}$, which depends on $\lambda\in ({\rm Im}\mathbb{H})_0^\mathbf{I}$, then we show that there exists a homeomorphism from $X(\lambda)/\mathbb{C}^{\times}$ to ${\rm Im}\mathbb{H}/\sim_{\lambda}$.

Put $Z_{\lambda}:=\{-\lambda_n\in {\rm Im}\mathbb{H};\ n\in\mathbf{I}\}$ for $\lambda\in ({\rm Im}\mathbb{H})_0^\mathbf{I}$. Then we have a disjoint union ${\rm Im}\mathbb{H} = Y_{\lambda}\bigsqcup Z_{\lambda}$.
\begin{definition}
{\rm Let $\lambda\in ({\rm Im}\mathbb{H})_0^\mathbf{I}$ and $\eta_1,\eta_2\in {\rm Im} \mathbb{H}$. We write $\eta_1\sim_{\lambda} \eta_2$ if they satisfy one of the following conditions; (i) $\eta_1$ and $\eta_2$ satisfy $\eta_{1,\mathbb{C}} = \eta_{2,\mathbb{C}}$ and $t\eta_1 + (1-t)\eta_2\in Y_{\lambda}$ for all $t\in[0,1]$, (ii) $\eta_1 = \eta_2 \in Z_{\lambda}$.}
\end{definition}
Now we obtain quotient spaces $X(\lambda)/\mathbb{C}^{\times}$ and ${\rm Im}\mathbb{H}/\sim_{\lambda}$ with quotient topology. Next we construct a homeomorphism between them.

Let $\mu_{\lambda}:X(\lambda) \to {\rm Im}\mathbb{H}$ be the hyper-K\"ahler moment map defined in Section 2. We will show that $\mu_{\lambda}$ induces a continuous map from $X(\lambda)/\mathbb{C}^{\times}$ to ${\rm Im}\mathbb{H}/\sim_{\lambda}$ by using the following lemma.
\begin{lem}\label{4.2}
Let $[z + w j]\in X(\lambda)^*$ and $g\in\mathbb{C}^{\times}$. Then we have $\mu_{\lambda}([z+wj])\sim_{\lambda}\mu_{\lambda}([z+wj]g)$ and
\begin{eqnarray}
\log |g|^2 = \int_{\mu_{\lambda,\mathbb{R}}([z+wj])}^{\mu_{\lambda,\mathbb{R}}([z+wj]g)}\Phi_{\lambda}(t, \zeta_{\mathbb{C}})dt,\nonumber
\end{eqnarray}
where $\Phi_{\lambda}$ is defined by
\begin{eqnarray}
\Phi_{\lambda}(\zeta) := \frac{1}{4}\sum_{n\in\mathbf{I}}\frac{1}{|\zeta + \lambda_n|}\nonumber
\end{eqnarray}
for $\zeta\in Y_{\lambda}$.

\end{lem}
\begin{proof}
Take $\tilde{g} = (\tilde{g}_n)_{n\in\mathbf{I}}\in (\mathbb{C}^{\times})^{\mathbf{I}}$ to be $\sum_{n\in\mathbf{I}}|1-\tilde{g}_n|^2 < \infty$ and $g = \prod_{n\in\mathbf{I}}\tilde{g}_n$. Now we regard $z + w j$ as an element of $\hat{\mu}_{\Lambda}^{-1}(0)$, and suppose $(z_n \tilde{g}_n + w_n \tilde{g}_n^{-1} j)\in \hat{\mu}_{\Lambda}^{-1}(0)$.

Put $\zeta = \mu_{\lambda}([z + wj])$ and $\eta = \mu_{\lambda}([z + wj]g)$. Then we have
\begin{eqnarray}
|z_n|^2 - |w_n|^2 &=& \lambda_{n,\mathbb{R}} + \zeta_{\mathbb{R}},\quad 2z_nw_n = \lambda_{n,\mathbb{C}} + \zeta_{\mathbb{C}},\nonumber\\
|z_n\tilde{g}_n|^2 - |w_n\tilde{g}_n^{-1}|^2 &=& \lambda_{n,\mathbb{R}} + \eta_{\mathbb{R}},\quad 2z_nw_n = \lambda_{n,\mathbb{C}} + \eta_{\mathbb{C}},\nonumber
\end{eqnarray}
accordingly we have $\zeta_{\mathbb{C}} = \eta_{\mathbb{C}}$. Since $[z + w j]\in X(\lambda)^*$, we may suppose $|z_n|^2 + |w_n|^2 \neq 0$ for all $n\in\mathbf{I}$. Then $\tilde{g}_n$ satisfies
\begin{eqnarray}
|\tilde{g}_n|^2 &=& \frac{|\eta + \lambda_n| + \eta_{\mathbb{R}} + \lambda_{n,\mathbb{R}}}{|\zeta + \lambda_n| + \zeta_{\mathbb{R}} + \lambda_{n,\mathbb{R}}}\quad ({\rm if}\ z_n \neq 0),\nonumber\\
|\tilde{g}_n|^{-2} &=& \frac{|\eta + \lambda_n| - (\eta_{\mathbb{R}} + \lambda_{n,\mathbb{R}})}{|\zeta + \lambda_n| - (\zeta_{\mathbb{R}} + \lambda_{n,\mathbb{R}})}\quad ({\rm if}\ w_n \neq 0).\nonumber
\end{eqnarray}
Now we put $\mathbf{I}_{\pm}(\zeta):=\{n\in\mathbf{I};\ \pm (\zeta_{\mathbb{R}} + \lambda_{n,\mathbb{R}}) > 0\}$. Since $|\tilde{g}_n|^2$ and $|\tilde{g}_n|^{-2}$ should be positive, we have $\eta = (\eta_{\mathbb{R}}, \zeta_{\mathbb{C}}) \sim_{\lambda} \zeta$. Then we obtain
\begin{eqnarray}
F_{\lambda}(\eta_{\mathbb{R}},\zeta_{\mathbb{R}},\zeta_{\mathbb{C}}) := \log |g|^2 &=& \sum_{n\in\mathbf{I}_+(\zeta)}\log \frac{|\eta + \lambda_n| + \eta_{\mathbb{R}} + \lambda_{n,\mathbb{R}}}{|\zeta + \lambda_n| + \zeta_{\mathbb{R}} + \lambda_{n,\mathbb{R}}}\nonumber\\
&\ & \quad  + \sum_{n\in\mathbf{I}_-(\zeta)}\log \frac{|\zeta + \lambda_n| - (\zeta_{\mathbb{R}} + \lambda_{n,\mathbb{R}})}{|\eta + \lambda_n| - (\eta_{\mathbb{R}} + \lambda_{n,\mathbb{R}})},\nonumber
\end{eqnarray}
where $\eta = \eta_{\mathbb{R}}i - \zeta_{\mathbb{C}}k$, then we have $\log |g|^2 = F_{\lambda}(\eta_{\mathbb{R}},\zeta_{\mathbb{R}},\zeta_{\mathbb{C}})$. The function $F_{\lambda}$ is smooth at $(\eta_{\mathbb{R}},\zeta_{\mathbb{R}},\zeta_{\mathbb{C}})$ if $\eta,\zeta\in Y_{\lambda}$. Then we have
\begin{eqnarray}
\frac{\partial F_{\lambda}}{\partial \eta_{\mathbb{R}}} = \Phi_{\lambda}(\eta_{\mathbb{R}},\zeta_{\mathbb{C}}) > 0.\nonumber
\end{eqnarray}
Since $F_{\lambda}(\zeta_{\mathbb{R}},\zeta_{\mathbb{R}},\zeta_{\mathbb{C}}) = 0$, we obtain
\begin{eqnarray}
\log |g|^2 = \int_{\zeta_{\mathbb{R}}}^{\eta_{\mathbb{R}}}\Phi_{\lambda}(t, \zeta_{\mathbb{C}})dt.\nonumber
\end{eqnarray}
\end{proof}
It is obvious that $[z+wj] = [z+wj]g$ if $\mu_{\lambda}([z+wj])\in Z_{\lambda}$. Then the hyper-K\"ahler moment map $\mu_{\lambda}$ induces $[\mu_{\lambda}]:X(\lambda)/\mathbb{C}^{\times} \to {\rm Im}\mathbb{H}/\sim_{\lambda}$ from Lemma \ref{4.2}. 
Since $\mu_{\lambda}$ is continuous and surjective, $[\mu_{\lambda}]$ is also continuous and surjective.
\begin{prop}
Let $\lambda\in ({\rm Im}\mathbb{H})_0^\mathbf{I}$ be generic. Then $[\mu_{\lambda}]:X(\lambda)/\mathbb{C}^{\times} \to {\rm Im}\mathbb{H}/\sim_{\lambda}$ is a homeomorphism.\label{homeo}
\end{prop}
\begin{proof}
It suffices to show that $[\mu_{\lambda}]$ is an injective and open map.

Let $[z + wj],[z' + w'j] \in X(\lambda)$ satisfy $\mu_{\lambda}([z + wj]) \sim_{\lambda} \mu_{\lambda}([z' + w'j])$. If $\mu_{\lambda}([z+wj])\in Z_{\lambda}$, then $[z + wj]=[z' + w'j]$. If $\mu_{\lambda}([z+wj])\in Y_{\lambda}$, then $\mu_{\lambda}([z' + w'j])$ is also an element of $Y_{\lambda}$. If we take $g\in \mathbb{C}^{\times}$ to be
\begin{eqnarray}
\log |g|^2 = \int_{\mu_{\lambda,\mathbb{R}}([z + wj])}^{\mu_{\lambda,\mathbb{R}}([z' + w'j])}\Phi_{\lambda}(t, \zeta_{\mathbb{C}})dt,\nonumber
\end{eqnarray}
then we have $\mu_{\lambda}([z + wj]g) = \mu_{\lambda}([z' + w'j])$. Since $S^1$ acts on  $\mu_{\lambda}^{-1}(\zeta)$ transitively for all $\zeta\in{\rm Im}\mathbb{H}$, there exists $\sigma\in S^1$ such that $[z + wj]g\sigma = [z' + w'j]$. Thus the injectivity has been proven.

The openness of $[\mu_{\lambda}]$ is easily shown by the elementary argument of general topology.
\end{proof}

From now on we identify $X(\lambda)/\mathbb{C}^{\times}$ with ${\rm Im}\mathbb{H}/\sim_{\lambda}$ by $[\mu_{\lambda}]$. To study the topological properties of $X(\lambda)/\mathbb{C}^{\times}$, we often observe ${\rm Im}\mathbb{H}/\sim_{\lambda}$ for convenience.

Now let $p_{\lambda}:X(\lambda) \to X(\lambda)/\mathbb{C}^{\times}$ and $\pi_{\lambda}:{\rm Im}\mathbb{H} \to {\rm Im}\mathbb{H}/\sim_{\lambda}$ be the quotient maps. 
Then $\mu_{\lambda,\mathbb{C}}:X(\lambda) \to \mathbb{C}$ induces a continuous map $[\mu_{\lambda,\mathbb{C}}]:X(\lambda)/\mathbb{C}^{\times} \to \mathbb{C}$ satisfying $[\mu_{\lambda,\mathbb{C}}]\circ p_{\lambda} = \mu_{\lambda,\mathbb{C}}$. 
On the other hand, the orthogonal projection ${\rm pr}_{\mathbb{C}}:{\rm Im}\mathbb{H} \to \mathbb{C}$ defined by ${\rm pr}_{\mathbb{C}}(\zeta):=\zeta_{\mathbb{C}}$ induces a continuous map $[{\rm pr}_{\mathbb{C}}]_{\lambda}:{\rm Im}\mathbb{H}/\sim_{\lambda} \to \mathbb{C}$ by $[{\rm pr}_{\mathbb{C}}]_{\lambda}\circ\pi_{\lambda} = {\rm pr}_{\mathbb{C}}$. 
Note that $[\mu_{\lambda,\mathbb{C}}]$ is identified with $[{\rm pr}_{\mathbb{C}}]_{\lambda}$ by $[\mu_{\lambda}]$, that is, $[\mu_{\lambda,\mathbb{C}}] = [{\rm pr}_{\mathbb{C}}]_{\lambda} \circ [\mu_{\lambda}]$.

There exists a natural partial order in ${\rm Im}\mathbb{H}/\sim_{\lambda}$ defined as follows.
\begin{definition}
{\rm For $\zeta, \eta\in{\rm Im}\mathbb{H}$, we write $\pi_{\lambda}(\zeta) \prec \pi_{\lambda}(\eta)$ if $\zeta_{\mathbb{C}} = \eta_{\mathbb{C}}$ and $\zeta_{\mathbb{R}} < \eta_{\mathbb{R}}$. Moreover we write $\pi_{\lambda}(\zeta) \preceq \pi_{\lambda}(\eta)$ if $\pi_{\lambda}(\zeta) \prec \pi_{\lambda}(\eta)$ or $\pi_{\lambda}(\zeta) = \pi_{\lambda}(\eta)$.}
\end{definition}
The above definition is well-defined and we have the structure of partially ordered set on ${\rm Im}\mathbb{H}/\sim_{\lambda}$.

\subsection{The topological structures of ${\rm Im}\mathbb{H}/\sim_{\lambda}$}
In this subsection we fix arbitrary generic $\lambda\in ({\rm Im}\mathbb{H})_0^\mathbf{I}$.

For an open set $V\subset \mathbb{C}$, put $\pi_{\lambda}(Y_{\lambda})|_{V}:=[{\rm pr}_{\mathbb{C}}]_{\lambda}^{-1}(V)\cap \pi_{\lambda}(Y_{\lambda})$ and
\begin{eqnarray}
\Gamma (V,\pi_{\lambda}(Y_{\lambda})|_{V}) := \{s:V\to \pi_{\lambda}(Y_{\lambda})|_{V};\ s\ {\rm is\ continuous,}\ [{\rm pr}_{\mathbb{C}}]_{\lambda}\circ s = {\rm id}_V\}.\nonumber
\end{eqnarray}
Here, the topology of $\mathbb{C}$ is the ordinary one as Euclidean space.
Under the identification ${\rm Im}\mathbb{H} = \mathbb{R}\times\mathbb{C}$ by $\zeta = \zeta_{\mathbb{R}}i - \zeta_{\mathbb{C}}k = (\zeta_{\mathbb{R}}, \zeta_{\mathbb{C}})$, all $s\in \Gamma (V,\pi_{\lambda}(Y_{\lambda})|_{V}) $ are written as $s(z) = \pi_{\lambda} ( \tilde{s}(z), z )$ for some continuous function $\tilde{s}:V \to \mathbb{R}$ such that the graph of $\tilde{s}$ does not intersect $Y_{\lambda}$.

Let $s_1,s_2:\mathbb{C}\to \pi_{\lambda}(Y_{\lambda})$ satisfy $[{\rm pr}_{\mathbb{C}}]_{\lambda}\circ s_1 = [{\rm pr}_{\mathbb{C}}]_{\lambda}\circ s_2 = {\rm id}_{\mathbb{C}}$, but are not necessary to be continuous, and put
\begin{eqnarray}
\mathbf{I}_{\lambda}^+(s_i) &:=& \{n\in\mathbf{I};\ \pi_{\lambda}(-\lambda_{n}) \prec s_i(-\lambda_{n,\mathbb{C}})\},\nonumber\\
\mathbf{I}_{\lambda}^-(s_i) &:=& \{n\in\mathbf{I};\ \pi_{\lambda}(-\lambda_{n}) \succ s_i(-\lambda_{n,\mathbb{C}})\}.\nonumber
\end{eqnarray}
Then we have a disjoint union $\mathbf{I} = \mathbf{I}_{\lambda}^+(s_i) \sqcup \mathbf{I}_{\lambda}^-(s_i)$. 
Then we define a map $k_{s_1,s_2}:\mathbb{C}\to \mathbb{Z}$ by
\begin{eqnarray}
k_{s_1,s_2}(z) := \sharp (\mathbf{I}_z\cap\mathbf{I}_{\lambda}^+(s_2)\cap\mathbf{I}_{\lambda}^-(s_1) ) - \sharp (\mathbf{I}_z\cap\mathbf{I}_{\lambda}^+(s_1)\cap\mathbf{I}_{\lambda}^-(s_2) )\nonumber
\end{eqnarray}
for $z\in\mathbb{C}$,
where $\mathbf{I}_z :=\{n\in\mathbf{I};\ -\lambda_{n,\mathbb{C}} = z\}$.
If $s_1,s_2$ are described as $s_i (z) = (\tilde{s}_i(z), z)$ for some $\tilde{s}_i \mathbb{C} \to \mathbb{R}$, we may write
\begin{eqnarray}
k_{s_1,s_2} (z) &=& \sharp \{n\in\mathbf{I}; \ -\lambda_{n,\mathbb{C}} = z,\ \tilde{s}_1(z) < -\lambda_{n,\mathbb{R}} < \tilde{s}_2(z)\} \nonumber\\
&\ &\quad - \sharp \{n\in\mathbf{I}; \ -\lambda_{n,\mathbb{C}} = z,\ \tilde{s}_2(z) < -\lambda_{n,\mathbb{R}} < \tilde{s}_1(z)\}. \nonumber
\end{eqnarray}
Now assume $s_1,s_2 \in  \Gamma (\mathbb{C},\pi_{\lambda}(Y_{\lambda}))$, hence $\tilde{s}_i$ can be taken as continuous functions.
Then the subset
\begin{eqnarray}
{\rm supp}(k_{s_1,s_2}) := \{z\in\mathbb{C};\ k_{s_1,s_2}(z) \neq 0\} \subset \mathbb{C}\nonumber
\end{eqnarray}
is discrete and closed because $\{ \lambda_n\in {\rm Im}\mathbb{H} ;\ n\in\mathbf{I}\} \subset {\rm Im}\mathbb{H}$ is also discrete and closed.

Conversely, let $s_1 \in  \Gamma (\mathbb{C},\pi_{\lambda}(Y_{\lambda}))$ and $s_2$ is not necessary to be continuous.
If ${\rm supp}(k_{s_1,s_2})$ is a discrete and closed subset of $\mathbb{C}$, then we can take $\tilde{s}_2$ to be continuous, consequently $s_2$ becomes continuous.
Thus we obtain the following proposition.

\begin{prop}
Let $s_1 \in \Gamma (\mathbb{C},\pi_{\lambda}(Y_{\lambda}))$. A map $s_2:\mathbb{C}\to \pi_{\lambda}(Y_{\lambda})$ which satisfies $[{\rm pr}_{\mathbb{C}}]_{\lambda}\circ s_2 = {\rm id}_{\mathbb{C}}$ is continuous if and only if ${\rm supp}(k_{s_1,s_2})$
is discrete and closed.
\label{4.11}
\end{prop}

\section{Biholomorphisms\label{main result}}
\subsection{Outline of the constructions}
In this subsection we explain how to construct biholomorphisms between $X(\lambda)$ and $X(\lambda')$ for some generic $\lambda,\lambda'\in ({\rm Im}\mathbb{H})_0^\mathbf{I}$. The biholomorphisms between $X(\lambda)$ and $X(\lambda')$ will be constructed if there exists a homeomorphism $\mathbf{h}:{\rm Im}\mathbb{H}/\sim_{\lambda} \to {\rm Im}\mathbb{H}/\sim_{\lambda'}$ preserving partial orders $\preceq$,
which satisfies $[{\rm pr}_{\mathbb{C}}]_{\lambda'}\circ\mathbf{h} = [{\rm pr}_{\mathbb{C}}]_{\lambda}$.

For each continuous section $s \in \Gamma (\mathbb{C},\pi_{\lambda}(Y_{\lambda}))$, we have an open subset
\begin{eqnarray}
X(\lambda)^s &:=& \mu_{\lambda}^{-1}(\pi_{\lambda}^{-1}(s(\mathbb{C})))\nonumber\\
&=& p_{\lambda}^{-1}([\mu_{\lambda}]^{-1}(s(\mathbb{C})))\subset X(\lambda),\nonumber
\end{eqnarray}
and it is easy to see $X(\lambda)^* = \bigcup_{s \in \Gamma (\mathbb{C},\pi_{\lambda}(Y_{\lambda}))}X(\lambda)^s$. In Section \ref{equiv hol func} the holomorphic coordinates over $X(\lambda)^s$ are constructed. By combining these holomorphic coordinates we obtain biholomorphic maps $X(\lambda)^s \to X(\lambda')^{\mathbf{h}(s)}$, then show that these glue on the intersections $X(\lambda)^{s_1}\cap X(\lambda)^{s_2}$ for all $s_1,s_2\in \Gamma (\mathbb{C},\pi_{\lambda}(Y_{\lambda}))$ in Section \ref{bihol}. Thus we obtain a biholomorphic map $X(\lambda)^* \to X(\lambda')^*$, which can be extended to a biholomorphic map $X(\lambda) \to X(\lambda')$.

\subsection{Holomorphic coordinates on $X(\lambda)^s$}\label{equiv hol func}
In this section we assume that $\lambda\in ({\rm Im}\mathbb{H})_0^\mathbf{I}$ is generic and $ \lambda_{n,\mathbb{R}} \neq 0$ for all $n\in\mathbf{I}$. We may assume the latter condition without loss of generality since there exists an isomorphism $X(\lambda) \cong X(\lambda + \underline{\eta})$ for all $\eta\in{\rm Im}\mathbb{H}$ from Section \ref{sec2.2}.

First of all we see that there exist $\mathbb{C}^{\times}$-equivariant holomorphic functions on $X(\lambda)^{\mathbf{o}_{\lambda}}$, where $\mathbf{o}_{\lambda}\in\Gamma (\mathbb{C},\pi_{\lambda}(Y_{\lambda}))$ is defined by $\mathbf{o}_{\lambda}(z):=\pi_{\lambda}(0,z)$. $\mathbf{I}_{\lambda}^{\pm}(\mathbf{o}_{\lambda})$ are given by
\begin{eqnarray}
\mathbf{I}_{\lambda}^+(\mathbf{o}_{\lambda}) &=& \{n\in\mathbf{I};\lambda_{n,\mathbb{R}} > 0\},\nonumber\\
\mathbf{I}_{\lambda}^-(\mathbf{o}_{\lambda}) &=& \{n\in\mathbf{I};\lambda_{n,\mathbb{R}} < 0\}.\nonumber
\end{eqnarray}
\begin{prop}\label{nonzero}
Let $[z + wj]\in X(\lambda)^{\mathbf{o}_{\lambda}}$. Then $z_n$ is nonzero if $n\in\mathbf{I}_{\lambda}^+(\mathbf{o}_{\lambda})$, and $w_n$ is nonzero if $n\in\mathbf{I}_{\lambda}^-(\mathbf{o}_{\lambda})$.
\end{prop}

\begin{proof}
We have ${\mu_{\lambda}}([z+wj]) \sim_{\lambda} (0,\mu_{\lambda,\mathbb{C}}([z+wj]))$ from the assumption $[z + wj]\in X(\lambda)^{\mathbf{o}_{\lambda}}$. 
By the injectivity of $[\mu_{\lambda}]$, there exists $g\in \mathbb{C}^{\times}$ such that $\mu_{\lambda}([z+wj]g) = (0,\mu_{\lambda,\mathbb{C}}([z+wj]))$. 
Thus we may suppose $\mu_{\lambda,\mathbb{R}}([z+wj]) = 0$, and we have $|z_n|^2 - |w_n|^2 = \lambda_{n,\mathbb{R}}$. Hence we obtain $|z_n|^2 > 0$ if $\lambda_{n,\mathbb{R}} > 0$, and $|w_n|^2 > 0$ if $\lambda_{n,\mathbb{R}} < 0$.
\end{proof}

Now we consider the infinite product
\begin{eqnarray}
\Big( \prod_{n\in\mathbf{I}_{\lambda}^+(\mathbf{o}_{\lambda})}\frac{z_n}{\alpha_n}\Big)
\Big( \prod_{n\in\mathbf{I}_{\lambda}^-(\mathbf{o}_{\lambda})}\frac{w_n}{\beta_n}\Big)^{-1}\label{hol.func}
\end{eqnarray}
for $z + wj\in \hat{\mu}_{\Lambda,\mathbb{C}}^{-1}(0)_{\lambda_{\mathbb{R}}}$ such that $[z + wj]\in X(\lambda)^{\mathbf{o}_{\lambda}}$, 
where we take $\Lambda\in \mathbb{H}^\mathbf{I}$ and $\alpha = (\alpha_n)_{n\in\mathbf{I}},\ \beta = (\beta_n)_{n\in\mathbf{I}}\in \mathbb{C}^{\mathbf{I}}$ to be $\Lambda_n i \overline{\Lambda}_n = \lambda_n$ and $\Lambda_n = \alpha_n + \beta_n j$. 
If we put $u_n:=z_n - \alpha_n$ and $v_n := w_n - \beta_n$, then we can see $\sum_{n\in\mathbf{I}}|u_n|^2 < +\infty$ and $\sum_{n\in\mathbf{I}}|v_n|^2 < +\infty$. 
On the other hand, we can deduce
\begin{eqnarray}
\sum_{n\in\mathbf{I}_{\lambda}^+(\mathbf{o}_{\lambda})}\frac{1}{|\alpha_n|^2} < +\infty,\ \sum_{n\in\mathbf{I}_{\lambda}^-(\mathbf{o}_{\lambda})}\frac{1}{|\beta_n|^2} < +\infty\nonumber
\end{eqnarray}
since $2|\alpha_n|^2 = |\lambda_n| + \lambda_{n,\mathbb{R}} \ge |\lambda_n|$ for $n\in\mathbf{I}_{\lambda}^+(\mathbf{o}_{\lambda})$, and $2|\beta_n|^2 = |\lambda_n| - \lambda_{n,\mathbb{R}} \ge |\lambda_n|$ for $n\in\mathbf{I}_{\lambda}^-(\mathbf{o}_{\lambda})$. Then the Cauchy-Schwarz inequality gives $\sum_{n\in\mathbf{I}_{\lambda}^+(\mathbf{o}_{\lambda})}\frac{|u_n|}{|\alpha_n|} < +\infty$ and $\sum_{n\in\mathbf{I}_{\lambda}^-(\mathbf{o}_{\lambda})}\frac{|v_n|}{|\beta_n|}< +\infty$, hence the infinite product (\ref{hol.func}) converges by the next lemma.

\begin{lem}\label{convergence}
Let $x_n\in \mathbb{C}\backslash \{ -1 \}$ for $n= 1,2,\cdots$. If we have $\sum_{n=1}^{\infty}|x_n| < +\infty$, then there exists a limit $\lim_{N\to\infty}\prod_{n=1}^N(1 + x_n) \neq 0$.
\end{lem}

\begin{proof}
Since $1 + x_n \neq 0$, we may put $1 + x_n = e^{a_n + b_ni}$ for some $a_n,b_n\in\mathbb{R}$ such that $-\pi < b_n \le \pi$.
Then we have $\prod_{n=1}^{N}(1 + x_n) = e^{\sum_{n=1}^{N} a_n + b_n i}$,
therefore it suffices to show the convergence of the series $\sum_{n=1}^{\infty} |a_n + b_ni|$.
From the assumption $\sum_{n=1}^{\infty}|x_n| < +\infty$, we may suppose there exists a sufficiently large positive integer $N_0$, and $|x_n| < \frac{1}{2}$ for all $n\ge N_0$. Then we have
\begin{eqnarray}
a_n + b_n i = \log (1 + x_n) = \sum_{k=1}^{\infty}(-1)^{n-1}\frac{x_n^k}{k}\nonumber
\end{eqnarray}
for every $n\ge N_0$. Consequently, we can deduce
\begin{eqnarray}
|a_n + b_n i| \le |x_n| \sum_{k=1}^{\infty} \frac{|x_n|^{k-1}}{k}
\le \Big( \sum_{k=1}^{\infty} \frac{1}{k2^{k-1}} \Big) |x_n|.\nonumber
\end{eqnarray}
Thus we obtain
\begin{eqnarray}
\sum_{n=1}^{\infty} |a_n + b_n i| \le \sum_{n=1}^{N_0} |a_n + b_n i| + \sum_{k=1}^{\infty} \frac{1}{k2^{k-1}} \sum_{n=N_0}^{\infty} |x_n| < +\infty.\nonumber
\end{eqnarray}
\end{proof}

From Proposition \ref{nonzero} and Lemma \ref{convergence}, the value of (\ref{hol.func}) is nonzero if $[z+wj]\in X_{GIT}(\lambda)^{\mathbf{o}_{\lambda}}$. Moreover the function (\ref{hol.func}) is $G_{\lambda}^{\mathbb{C}}$-invariant, consequently, it induces a smooth function $f_{\lambda}^{\mathbf{o}_{\lambda}}: X(\lambda)^{\mathbf{o}_{\lambda}} \to \mathbb{C}^{\times}$ defined by
\begin{eqnarray}
f_{\lambda}^{\mathbf{o}_{\lambda}}([z+wj]) := \prod_{n\in\mathbf{I}_{\lambda}^+(\mathbf{o}_{\lambda})}\frac{z_n}{\alpha_n} \cdot \Big( \prod_{n\in\mathbf{I}_{\lambda}^-(\mathbf{o}_{\lambda})}\frac{w_n}{\beta_n}\Big)^{-1}\nonumber
\end{eqnarray}
for $[z+wj] \in X_{GIT}(\lambda)^{\mathbf{o}_{\lambda}}$. It is easy to check that $f_{\lambda}^{\mathbf{o}_{\lambda}}$ is $\mathbb{C}^{\times}$-equivariant, in the sense $f_{\lambda}^{\mathbf{o}_{\lambda}}([z+wj]g) = gf_{\lambda}^{\mathbf{o}_{\lambda}}([z+wj])$ for all $g\in\mathbb{C}^{\times}$.

\begin{prop}\label{darboux}
On $X(\lambda)^{\mathbf{o}_{\lambda}}$, the holomorphic symplectic form is given by
\begin{eqnarray}
2 \omega_{\lambda,\mathbb{C}} = \frac{d f_{\lambda}^{\mathbf{o}_{\lambda}}}{f_{\lambda}^{\mathbf{o}_{\lambda}}} \wedge d\mu_{\lambda,\mathbb{C}} \nonumber
\end{eqnarray}
\end{prop}

\begin{proof}
Let $\iota_{\Lambda}: \hat{\mu}_{\Lambda,\mathbb{C}}^{-1}(0)_{\lambda_{\mathbb{R}}} \to M_{\Lambda}$ be the embedding map, and $\pi_{\Lambda}: \hat{\mu}_{\Lambda,\mathbb{C}}^{-1}(0)_{\lambda_{\mathbb{R}}} \to X_{GIT}(\lambda)$ be the quotient map.
Since $\omega_{\lambda,\mathbb{C}}$ is defined by $\pi_{\Lambda}^*\omega_{\lambda,\mathbb{C}} = \iota_{\Lambda}^* \sum_{n\in\mathbf{I}} dz_n\wedge dw_n$, we have
\begin{eqnarray}
\pi_{\Lambda}^* (d f_{\lambda}^{\mathbf{o}_{\lambda}}\wedge d\mu_{\lambda,\mathbb{C}})_{z+wj} &=& d\Big\{ \prod_{n\in\mathbf{I}_{\lambda}^+(\mathbf{o}_{\lambda})}\frac{z_n}{\alpha_n} \cdot \Big( \prod_{n\in\mathbf{I}_{\lambda}^-(\mathbf{o}_{\lambda})}\frac{w_n}{\beta_n}\Big)^{-1}\Big\} \wedge (d\mu_{\lambda,\mathbb{C}})_{[z+wj]}\nonumber\\
&=& f_{\lambda}^{\mathbf{o}_{\lambda}}([z+wj]) \bigg( \sum_{n\in\mathbf{I}_{\lambda}^+(\mathbf{o}_{\lambda})} \frac{dz_n}{z_n}\wedge d(2z_nw_n ) \nonumber\\
&\ & \quad\quad\quad\quad\quad\quad\quad\quad - \sum_{n\in\mathbf{I}_{\lambda}^-(\mathbf{o}_{\lambda})} \frac{dw_n}{w_n}\wedge d(2z_nw_n ) \bigg) \nonumber\\
&=& 2 f_{\lambda}^{\mathbf{o}_{\lambda}}([z+wj]) \sum_{n\in\mathbf{I}}  dz_n\wedge dw_n.\nonumber
\end{eqnarray}
Here we use $\mu_{\lambda,\mathbb{C}}([z+wj]) = 2z_nw_n -\lambda_{n,\mathbb{C}}$ for any $n\in\mathbf{I}$.
\end{proof}

The next lemma may be well-known, but we show it for the reader's convenience.
\begin{lem}
Let $U$ be a complex manifold of dimension $n$ and $f_1,\cdots,f_n\in C^{\infty}(U)$. If $df_1\wedge\cdots\wedge df_n\in\Omega^{n,0}(U)$ and $df_1\wedge\cdots\wedge df_n|_p\neq 0$ for all $p\in U$, then $(f_1,\cdots,f_n):U\to \mathbb{C}^n$ is a local biholomorphism. \label{localchart}
\end{lem}
\begin{proof}
Since $df_1\wedge\cdots\wedge df_n$ is in $\Omega^{(n,0)}(U)$ and never be zero, we have
\begin{eqnarray}
df_1\wedge\cdots\wedge df_n = \partial f_1\wedge\cdots\wedge \partial f_n \neq 0.\nonumber
\end{eqnarray}
Therefore $\partial f_1,\cdots, \partial f_n$ becomes a basis of $(T^*_pU)^{(1,0)}$ for all $p\in U$. Since $(n-1,1)$-part of $df_1\wedge\cdots\wedge df_n$ vanishes,
we have $\overline{\partial} f_i = 0$. Then $(f_1,\cdots,f_n):U\to \mathbb{C}^n$ is locally biholomorphic since the Jacobian is everywhere invertible because $\partial f_1\wedge\cdots\wedge \partial f_n \neq 0$.
\end{proof}

From Lemma \ref{localchart}, we obtain a local holomorphic chart
\begin{eqnarray}
(f_{\lambda}^{\mathbf{o}_{\lambda}}, \mu_{\lambda,\mathbb{C}}):X(\lambda)^{\mathbf{o}_{\lambda}} \to \mathbb{C}^{\times}\times\mathbb{C}.\nonumber
\end{eqnarray}
To show that $(f_{\lambda}^{\mathbf{o}_{\lambda}}, \mu_{\lambda,\mathbb{C}})$ is biholomorphic, it suffices to show that the map is bijective. We will show it later.

Next we consider $\mathbb{C}^{\times}$-equivariant holomorphic functions over $X(\lambda)^s$ for an arbitrary $s\in\Gamma (\mathbb{C},\pi_{\lambda}(Y_{\lambda}))$.

Take a map $k:\mathbb{C}\to\mathbb{Z}$ such that ${\rm supp}(k) = k^{-1}(\mathbb{Z}\backslash\{ 0\})\subset\mathbb{C}$ is discrete and closed. Then denote by $\mathcal{A}(k)$ the subset of all meromorphic functions on $\mathbb{C}$, which consists of the meromorphic functions $\varphi$ who have the limits
\begin{eqnarray}
\lim_{w\to z}\varphi(w)(w-z)^{ - k(z)}\in\mathbb{C}^{\times}\nonumber
\end{eqnarray}
for all $z\in {\rm supp}(k)$. Then $\varphi$ is a $\mathbb{C}^{\times}$ valued holomorphic function on $k^{-1}(0)$.

Now we put
\begin{eqnarray}
f_{\lambda}^{\mathbf{o}_{\lambda},\varphi}(p):=f_{\lambda}^{\mathbf{o}_{\lambda}}(p)\cdot\varphi(\mu_{\lambda,\mathbb{C}}(p))\nonumber
\end{eqnarray}
for $s\in \Gamma (\mathbb{C},\pi_{\lambda}(Y_{\lambda}))$ and $\varphi\in\mathcal{A}(k_{\mathbf{o}_{\lambda},s})$, which is a $\mathbb{C}^{\times}$-valued holomorphic function on $X(\lambda)^{\mathbf{o}_{\lambda}} \cap X(\lambda)^s$.
\begin{prop}
Let $s\in \Gamma (\mathbb{C},\pi_{\lambda}(Y_{\lambda}))$ and $\varphi\in\mathcal{A}(k_{\mathbf{o}_{\lambda},s})$. Then $f_{\lambda}^{\mathbf{o}_{\lambda},\varphi}$ extends to $\mathbb{C}^{\times}$-equivariant holomorphic map $X(\lambda)^s \to \mathbb{C}^{\times}$. \label{5.2}
\end{prop}
\begin{proof}
Since $f_{\lambda}^{\mathbf{o}_{\lambda},\varphi}$ can be regarded as a $\mathbb{C}^{\times}$-equivariant holomorphic map $X(\lambda)^{\mathbf{o}_{\lambda}} \cap X(\lambda)^s \to \mathbb{C}^{\times}$, it suffices to show that $f_{\lambda}^{\mathbf{o}_{\lambda},\varphi}$ is extended to $X(\lambda)^s$ continuously.

Let $[z+wj]\in X_{GIT}(\lambda)^{\mathbf{o}_{\lambda}} \cap X_{GIT}(\lambda)^s$.
We fix $m\in\mathbf{I}$ arbitrarily, and put $\hat{Z}_{A}:= \prod_{n\in A}\frac{z_n}{\alpha_n},\ \hat{W}_{A}:= \prod_{n\in A}\frac{w_n}{\beta_n}$ for $A\subset \mathbf{I}$.
First of all, the following conditions are all equivalent for all $m\in\mathbf{I}$ and $s_1,s_2\in \Gamma (\mathbb{C},\pi_{\lambda}(Y_{\lambda}))$; $(i)\ s_1(-\lambda_{m,\mathbb{C}}) \preceq s_2(-\lambda_{m,\mathbb{C}})$, $(ii)$ $\mathbf{I}_{\lambda}^+(s_1)\cap\mathbf{I}_{-\lambda_{m,\mathbb{C}}} \subset \mathbf{I}_{\lambda}^+(s_2)\cap\mathbf{I}_{-\lambda_{m,\mathbb{C}}}$, $(iii)$ $\mathbf{I}_{\lambda}^-(s_2)\cap\mathbf{I}_{-\lambda_{m,\mathbb{C}}} \subset \mathbf{I}_{\lambda}^-(s_1)\cap\mathbf{I}_{-\lambda_{m,\mathbb{C}}}$.

Assume $s(-\lambda_{m,\mathbb{C}}) \preceq \mathbf{o}_{\lambda}(-\lambda_{m,\mathbb{C}})$. Then we can deduce
\begin{eqnarray}
f_{\lambda}^{\mathbf{o}_{\lambda}}([z+wj])
 &=& \hat{Z}_{\mathbf{I}_{\lambda}^+(\mathbf{o}_{\lambda})} \hat{W}_{\mathbf{I}_{\lambda}^-(\mathbf{o}_{\lambda})}^{-1}\nonumber\\
 &=& \frac{ \hat{Z}_{\mathbf{I}_{\lambda}^+(\mathbf{o}_{\lambda})\cap\mathbf{I}_{-\lambda_{m,\mathbb{C}}}}
 \hat{Z}_{\mathbf{I}_{\lambda}^+(\mathbf{o}_{\lambda})\backslash\mathbf{I}_{-\lambda_{m,\mathbb{C}}}} }
{ \hat{W}_{\mathbf{I}_{\lambda}^-(\mathbf{o}_{\lambda})\cap\mathbf{I}_{-\lambda_{m,\mathbb{C}}}}
 \hat{W}_{\mathbf{I}_{\lambda}^-(\mathbf{o}_{\lambda})\backslash\mathbf{I}_{-\lambda_{m,\mathbb{C}}}} }
 \nonumber\\
 &=& \frac{ \hat{Z}_{ ( \mathbf{I}_{\lambda}^+(\mathbf{o}_{\lambda}) \backslash \mathbf{I}_{\lambda}^+(s) ) \cap\mathbf{I}_{-\lambda_{m,\mathbb{C}}} }
 \hat{Z}_{\mathbf{I}_{\lambda}^+(s)\cap\mathbf{I}_{-\lambda_{m,\mathbb{C}}}}
 \hat{Z}_{\mathbf{I}_{\lambda}^+(\mathbf{o}_{\lambda}) \backslash \mathbf{I}_{-\lambda_{m,\mathbb{C}}}} }
{ \hat{W}_{( \mathbf{I}_{\lambda}^-(s) \backslash \mathbf{I}_{\lambda}^-(\mathbf{o}_{\lambda}) )\cap\mathbf{I}_{-\lambda_{m,\mathbb{C}}} }^{-1}
\hat{W}_{\mathbf{I}_{\lambda}^-(s)\cap\mathbf{I}_{-\lambda_{m,\mathbb{C}}}}
 \hat{W}_{\mathbf{I}_{\lambda}^-(\mathbf{o}_{\lambda})\backslash\mathbf{I}_{-\lambda_{m,\mathbb{C}}}} }
 \nonumber\\
 &=& \frac{ \hat{Z}_{\mathbf{I}_{\lambda}^+(s)\cap\mathbf{I}_{-\lambda_{m,\mathbb{C}}}}
 \hat{Z}_{\mathbf{I}_{\lambda}^+(\mathbf{o}_{\lambda}) \backslash \mathbf{I}_{-\lambda_{m,\mathbb{C}}}} }
{ \hat{W}_{\mathbf{I}_{\lambda}^-(s)\cap\mathbf{I}_{-\lambda_{m,\mathbb{C}}}}
 \hat{W}_{\mathbf{I}_{\lambda}^-(\mathbf{o}_{\lambda})\backslash\mathbf{I}_{-\lambda_{m,\mathbb{C}}}} }
 \times
 \prod_{n\in ( \mathbf{I}_{\lambda}^+(\mathbf{o}_{\lambda}) \backslash \mathbf{I}_{\lambda}^+(s) ) \cap\mathbf{I}_{-\lambda_{m,\mathbb{C}}}} \frac{z_nw_n}{\alpha_n\beta_n}.\nonumber
\end{eqnarray}
Here we use $\mathbf{I}_{\lambda}^+(\mathbf{o}_{\lambda}) \backslash \mathbf{I}_{\lambda}^+(s) = \mathbf{I}_{\lambda}^-(s) \backslash \mathbf{I}_{\lambda}^-(\mathbf{o}_{\lambda})$ for the last equality.  Now we put $\zeta_{\mathbb{C}} = \mu_{\lambda,\mathbb{C}}([z+wj])$. Then we have $\zeta_{\mathbb{C}} = 2z_nw_n - \lambda_{n,\mathbb{C}}$ and $2\alpha_n\beta_n = \lambda_{n,\mathbb{C}}$, hence
\begin{eqnarray}
\frac{z_nw_n}{\alpha_n\beta_n} = \frac{\zeta_{\mathbb{C}} + \lambda_{m,\mathbb{C}} }{ \lambda_{m,\mathbb{C}} }\nonumber
\end{eqnarray}
if $n\in \mathbf{I}_{-\lambda_{m,\mathbb{C}}}$.
Thus we obtain
\begin{eqnarray}
f_{\lambda}^{\mathbf{o}_{\lambda}}([z+wj]) = \Big( \frac{\zeta_{\mathbb{C}} + \lambda_{m,\mathbb{C}}}{2}\Big)^{ - k_{\mathbf{o}_{\lambda},s}(- \lambda_{m,\mathbb{C}}) } 
\frac{ \hat{Z}_{\mathbf{I}_{\lambda}^+(s)\cap\mathbf{I}_{-\lambda_{m,\mathbb{C}}}}
 \hat{Z}_{\mathbf{I}_{\lambda}^+(\mathbf{o}_{\lambda}) \backslash \mathbf{I}_{-\lambda_{m,\mathbb{C}}}} }
{ \hat{W}_{\mathbf{I}_{\lambda}^-(s)\cap\mathbf{I}_{-\lambda_{m,\mathbb{C}}}}
 \hat{W}_{\mathbf{I}_{\lambda}^-(\mathbf{o}_{\lambda})\backslash\mathbf{I}_{-\lambda_{m,\mathbb{C}}}} }.
 \nonumber
\end{eqnarray}
Since
\begin{eqnarray}
\varphi(\zeta_{\mathbb{C}})\Big( \frac{\zeta_{\mathbb{C}} + \lambda_{m,\mathbb{C}}}{2}\Big)^{ - k_{\mathbf{o}_{\lambda},s}(- \lambda_{m,\mathbb{C}})}\nonumber
\end{eqnarray}
is $\mathbb{C}^{\times}$-valued holomorphic on the neighborhood of $\zeta_{\mathbb{C}} = - \lambda_{m,\mathbb{C}}$, and
\begin{eqnarray}
\frac{ \hat{Z}_{\mathbf{I}_{\lambda}^+(s)\cap\mathbf{I}_{-\lambda_{m,\mathbb{C}}}}
 \hat{Z}_{\mathbf{I}_{\lambda}^+(\mathbf{o}_{\lambda}) \backslash \mathbf{I}_{-\lambda_{m,\mathbb{C}}}} }
{ \hat{W}_{\mathbf{I}_{\lambda}^-(s)\cap\mathbf{I}_{-\lambda_{m,\mathbb{C}}}}
 \hat{W}_{\mathbf{I}_{\lambda}^-(\mathbf{o}_{\lambda})\backslash\mathbf{I}_{-\lambda_{m,\mathbb{C}}}} }\nonumber
\end{eqnarray}
is also $\mathbb{C}^{\times}$-valued at $\zeta_{\mathbb{C}} = - \lambda_{m,\mathbb{C}}$, then $f_{\lambda}^{\mathbf{o}_{\lambda},\varphi}$ can be extended continuously to $\mu_{\lambda}^{-1}(\pi_{\lambda}^{-1}(s(-\lambda_{m,\mathbb{C}})))$ for each $m\in\mathbf{I}$, accordingly extended to $X(\lambda)^s$.
\end{proof}

Now $(f_{\lambda}^{\mathbf{o}_{\lambda},\varphi},\mu_{\lambda,\mathbb{C}}): X(\lambda)^s\to\mathbb{C}^{\times}\times\mathbb{C}$ is locally biholomorphic since we have
\begin{eqnarray}
d f_{\lambda}^{\mathbf{o}_{\lambda},\varphi}\wedge d\mu_{\lambda,\mathbb{C}} = 2 f_{\lambda}^{\mathbf{o}_{\lambda},\varphi}\omega_{\lambda,\mathbb{C}}\nonumber
\end{eqnarray}
on $X(\lambda)^s$. The above equation follows from Proposition \ref{darboux} and the definition of $f_{\lambda}^{\mathbf{o}_{\lambda},\varphi}$.
\begin{prop}
Let $s\in \Gamma (\mathbb{C},\pi_{\lambda}(Y_{\lambda}))$ and $\varphi\in\mathcal{A}(k_{\mathbf{o}_{\lambda},s})$. Then
\begin{eqnarray}
(f_{\lambda}^{\mathbf{o}_{\lambda},\varphi},\mu_{\lambda,\mathbb{C}}): X(\lambda)^s\to\mathbb{C}^{\times}\times\mathbb{C}\nonumber
\end{eqnarray}
is biholomorphic.\label{5.4}
\end{prop}
\begin{proof}
Since we have shown that the map is locally biholomorphic, it suffices to show that it is bijective.

First of all we show the injectivity. Let $[z+wj],[z'+w'j]\in X_{GIT}(\lambda)^s$ satisfy
\begin{eqnarray}
\mu_{\lambda,\mathbb{C}}([z+wj]) &=& \mu_{\lambda,\mathbb{C}}([z'+w'j]),\label{cpxmm}\\
f_{\lambda}^{\mathbf{o}_{\lambda},\varphi}([z+wj]) &=& f_{\lambda}^{\mathbf{o}_{\lambda},\varphi}([z'+w'j]).\label{c*equiv}
\end{eqnarray}
Then (\ref{cpxmm}) gives that $\pi_{\lambda}(\mu_{\lambda}([z+wj])) = \pi_{\lambda}(\mu_{\lambda}([z'+w'j]))$. From Proposition \ref{homeo}, there exists $g\in\mathbb{C}^{\times}$ such that $[z+wj]g = [z'+w'j]$. Therefore we have $[z+wj] = [z'+w'j]$ since $f_{\lambda}^{\mathbf{o}_{\lambda},\varphi}$ is $\mathbb{C}^{\times}$-equivariant and $\mathbb{C}^{\times}$-valued,which gives $g=1$.

Next we show the surjectivity. Take $(p,q)\in\mathbb{C}^{\times}\times\mathbb{C}$ arbitrarily. Fix $[z+wj]\in\mu_{\lambda}^{-1}(\pi_{\lambda}^{-1}(s(q)))$. If we put $g_0:=f_{\lambda}^{\mathbf{o}_{\lambda},\varphi}([z+wj])$, then $[z+wj]g_0^{-1}p\in X_{GIT}(\lambda)^s$ satisfies $f_{\lambda}^{\mathbf{o}_{\lambda},\varphi}([z+wj]g_0^{-1}p) = p$ and $\mu_{\lambda,\mathbb{C}}([z+wj]g_0^{-1}p) = \mu_{\lambda,\mathbb{C}}([z+wj]) = q$.
\end{proof}

For all $s\in \Gamma (\mathbb{C},\pi_{\lambda}(Y_{\lambda}))$, $\mathcal{A}(k_{\mathbf{o}_{\lambda},s})$ is not empty from Weierstrass Theorem. If we put $\mathcal{G}:=\{f:\mathbb{C}\to\mathbb{C}^{\times}\ {\rm is\ holomorphic}\} = \Gamma(\mathbb{C},\mathcal{O}_{\mathbb{C}}^{\times})$, then $\mathcal{G}$ acts on $\mathcal{A}(k_{\mathbf{o}_{\lambda},s})$ transitively and freely.

Next we consider the gluing.
Take $s_1, s_2\in \Gamma (\mathbb{C},\pi_{\lambda}(Y_{\lambda}))$, and $\varphi_i\in \mathcal{A}(k_{\mathbf{o}_{\lambda},s_i})$ for $i=1,2$. We put $F_{\lambda}^{\varphi}:=(f_{\lambda}^{\mathbf{o}_{\lambda},\varphi},\mu_{\lambda,\mathbb{C}})$ and define
\begin{eqnarray}
\psi_{\lambda}^{\varphi_2,\varphi_1}:F_{\lambda}^{\varphi_1}(X(\lambda)^{s_1}\cap X(\lambda)^{s_2})\to F_{\lambda}^{\varphi_2}(X(\lambda)^{s_1}\cap X(\lambda)^{s_2})\nonumber
\end{eqnarray}
by $\psi_{\lambda}^{\varphi_2,\varphi_1} := F_{\lambda}^{\varphi_2}\circ (F_{\lambda}^{\varphi_1})^{-1}$. Now we take $p\in\mathbb{C}^{\times}$ and $q\in\mathbb{C}$ to be $(p,q)\in F_{\lambda}^{\varphi_1}(X(\lambda)^{s_1}\cap X(\lambda)^{s_2})$. Since we have
\begin{eqnarray}
F_{\lambda}^{\varphi_2} &=& (f_{\lambda}^{\mathbf{o}_{\lambda},\varphi_2},\mu_{\lambda,\mathbb{C}})\nonumber\\
&=& (f_{\lambda}^{\mathbf{o}_{\lambda}}\cdot\varphi_1(\mu_{\lambda,\mathbb{C}})\cdot\frac{\varphi_2(\mu_{\lambda,\mathbb{C}})}{\varphi_1(\mu_{\lambda,\mathbb{C}})},\mu_{\lambda,\mathbb{C}})\nonumber\\
&=& (f_{\lambda}^{\mathbf{o}_{\lambda},\varphi_1}\cdot\frac{\varphi_2(\mu_{\lambda,\mathbb{C}})}{\varphi_1(\mu_{\lambda,\mathbb{C}})},\mu_{\lambda,\mathbb{C}}),\nonumber
\end{eqnarray}
then we can write as
\begin{eqnarray}
\psi_{\lambda}^{\varphi_2,\varphi_1}(p,q) = (p\cdot\frac{\varphi_2(q)}{\varphi_1(q)}, q).\label{gluing formula}
\end{eqnarray}
Consequently, we have
\begin{eqnarray}
F_{\lambda}^{\varphi_1}(X(\lambda)^{s_1}\cap X(\lambda)^{s_2}) = F_{\lambda}^{\varphi_2}(X(\lambda)^{s_1}\cap X(\lambda)^{s_2}) = \mathbb{C}^{\times} \times k_{s_1,s_2}^{-1}(\{ 0\}).\nonumber
\end{eqnarray}

\subsection{The construction of biholomorphic maps\label{bihol}}
Recall that we have put $X(\lambda)^* = \mu_{\lambda}^{-1}(Y_{\lambda})$. In this section we construct biholomorphisms between $X(\lambda)^*$ and $X(\lambda')^*$ for $\lambda,\lambda'\in ({\rm Im}\mathbb{H})_0^\mathbf{I}$, which satisfy appropriate conditions. First of we describe these conditions for $\lambda,\lambda'\in ({\rm Im}\mathbb{H})_0^\mathbf{I}$.

Let $\lambda,\lambda'\in ({\rm Im}\mathbb{H})_0^\mathbf{I}$ be generic.
Then we denote by ${\rm Isom}(\lambda,\lambda')$ the set which consists of all homeomorphisms $\mathbf{h}:{\rm Im}\mathbb{H}/\sim_{\lambda} \to {\rm Im}\mathbb{H}/\sim_{\lambda'}$ preserving partial orders $\preceq$,
which satisfies $[{\rm pr}_{\mathbb{C}}]_{\lambda'}\circ\mathbf{h} = [{\rm pr}_{\mathbb{C}}]_{\lambda}$.
We can construct a $\mathbb{C}^{\times}$-equivariant biholomorphism from $X(\lambda)^*$ and $X(\lambda')^*$ which preserves holomorphic symplectic forms $\omega_{\lambda,\mathbb{C}}$ and $\omega_{\lambda',\mathbb{C}}$ as follows.

Let $\lambda,\lambda'\in ({\rm Im}\mathbb{H})_0^\mathbf{I}$ be generic, and $\mathbf{h}\in {\rm Isom}(\lambda,\lambda')$.
Then $\mathbf{h}$ induces a one-to-one correspondence $\Gamma (\mathbb{C},\pi_{\lambda}(Y_{\lambda}))\to\Gamma (\mathbb{C},\pi_{\lambda'}(Y_{\lambda'}))$ which we use the same symbol $\mathbf{h}:\Gamma (\mathbb{C},\pi_{\lambda}(Y_{\lambda})) \to \Gamma (\mathbb{C},\pi_{\lambda'}(Y_{\lambda'}))$. We may assume that $ \lambda_{n,\mathbb{R}} \neq 0$ and $ \lambda'_{n,\mathbb{R}} \neq 0$ for all $n\in\mathbf{I}$ without loss of generality. Then ${\rm supp} (k_{\mathbf{o}_{\lambda'},\mathbf{h}(\mathbf{o}_{\lambda})})$ becomes discrete and closed from Proposition \ref{4.11}, accordingly we can take $\varphi_0\in \mathcal{A}(k_{\mathbf{o}_{\lambda'},\mathbf{h}(\mathbf{o}_{\lambda})})$ since $\mathcal{A}(k_{\mathbf{o}_{\lambda'},\mathbf{h}(\mathbf{o}_{\lambda})})$ is not empty by Weierstrass Theorem.

To construct biholomorphisms from $X(\lambda)^*$ to $X(\lambda')^*$, it suffices to construct biholomorphisms from $X(\lambda)^s$ to $X(\lambda')^{\mathbf{h}(s)}$ and glue them since a family of open sets $\{ X(\lambda)^s\}_{s\in\Gamma (\mathbb{C},\pi_{\lambda}(Y_{\lambda}))}$ is an open covering of $X(\lambda)^*$.

Recall that $F_{\lambda}^{\varphi}:X(\lambda)^s \to \mathbb{C}^{\times} \times \mathbb{C}$ is a biholomorphism for each $\varphi\in \mathcal{A}(k_{\mathbf{o}_{\lambda},s})$.
Now we have $k_{\mathbf{o}_{\lambda},s} = k_{\mathbf{h}(\mathbf{o}_{\lambda}),\mathbf{h}(s)}$ since $\mathbf{h}$ preserves the partial orders, hence $\varphi\varphi_0$ is an element of $\mathcal{A}(k_{\mathbf{o}_{\lambda'},\mathbf{h}(s)})$.
Consequently, we have a biholomorphism $F_{\lambda'}^{\varphi \varphi_0}:X(\lambda')^{\mathbf{h}(s)} \to \mathbb{C}^{\times} \times \mathbb{C}$, then a biholomorphism $H_{s,\varphi}(\mathbf{h},\varphi_0):X(\lambda)^s \to X(\lambda')^{\mathbf{h}(s)}$ is obtained by
\begin{eqnarray}
H_{s,\varphi}(\mathbf{h},\varphi_0) := (F_{\lambda'}^{\varphi \varphi_0})^{-1}\circ F_{\lambda}^{\varphi}\nonumber
\end{eqnarray}
for $s\in\Gamma (\mathbb{C},\pi_{\lambda}(Y_{\lambda}))$, $\varphi_0\in \mathcal{A}(k_{\mathbf{o}_{\lambda'},\mathbf{h}(\mathbf{o}_{\lambda})})$ and $\varphi\in \mathcal{A}(k_{\mathbf{o}_{\lambda},s})$.
\begin{prop}
Let $\lambda,\lambda'\in ({\rm Im}\mathbb{H})_0^\mathbf{I}$ be generic. Then $H_{s_1,\varphi_1}(\mathbf{h},\varphi_0)$ and $H_{s_2,\varphi_2}(\mathbf{h},\varphi_0)$ are glued on $X(\lambda)^{s_1}\cap X(\lambda)^{s_2}$ for all $\mathbf{h}\in {\rm Isom}(\lambda,\lambda')$, $s_1,s_2\in\Gamma (\mathbb{C},\pi_{\lambda}(Y_{\lambda}))$, $\varphi_0\in \mathcal{A}(k_{\mathbf{o}_{\lambda'},\mathbf{h}(\mathbf{o}_{\lambda})})$ and $\varphi_i\in \mathcal{A}(k_{\mathbf{o}_{\lambda},s_i})$.
\end{prop}
\begin{proof}
Recall that we have $F_{\lambda}^{\varphi_1} = \psi_{\lambda}^{\varphi_1,\varphi_2}\circ F_{\lambda}^{\varphi_2}$ on $U := X(\lambda)^{s_1}\cap X(\lambda)^{s_2}$ and $F_{\lambda}^{\varphi_1}(U) = F_{\lambda}^{\varphi_2}(U)$. By the definition of $H_{s_1,\varphi_1}(\mathbf{h},\varphi_0)$, we can see
\begin{eqnarray}
H_{s_1,\varphi_1}(\mathbf{h},\varphi_0)|_{U} &=& (F_{\lambda'}^{\varphi_1 \varphi_0})^{-1}\circ F_{\lambda}^{\varphi_1} |_{U}\nonumber\\
&=& (\psi_{\lambda'}^{\varphi_1\varphi_0,\varphi_2\varphi_0}\circ F_{\lambda'}^{\varphi_2 \varphi_0})^{-1}\circ \psi_{\lambda}^{\varphi_1,\varphi_2}\circ F_{\lambda}^{\varphi_2} |_{U}\nonumber\\
&=& (F_{\lambda'}^{\varphi_2 \varphi_0})^{-1}\circ (\psi_{\lambda'}^{\varphi_1\varphi_0,\varphi_2\varphi_0})^{-1}\circ \psi_{\lambda}^{\varphi_1,\varphi_2}\circ F_{\lambda}^{\varphi_2} |_{U}.\nonumber
\end{eqnarray}
For each $(p,q)\in F_{\lambda}^{\varphi_1}(U) = F_{\lambda}^{\varphi_2}(U)$, we have
\begin{eqnarray}
(\psi_{\lambda'}^{\varphi_1\varphi_0,\varphi_2\varphi_0})^{-1}\circ \psi_{\lambda}^{\varphi_1,\varphi_2}(p,q) &=& (\psi_{\lambda'}^{\varphi_1\varphi_0,\varphi_2\varphi_0})^{-1} \Big( p \cdot \frac{\varphi_1(q)}{\varphi_2(q)},\ q\Big)\nonumber\\
&=& \Big( p \cdot \frac{\varphi_1(q)}{\varphi_2(q)} \cdot \frac{\varphi_2(q)\varphi_0(q)}{\varphi_1(q)\varphi_0(q)},\ q \Big) = (p,q),\nonumber
\end{eqnarray}
which gives
\begin{eqnarray}
H_{s_1,\varphi_1}(\mathbf{h},\varphi_0)|_{U}
&=& (F_{\lambda'}^{\varphi_2 \varphi_0})^{-1}\circ {\rm id }_{F_{\lambda}^{\varphi_2}(U)}\circ F_{\lambda}^{\varphi_2} |_{U}\nonumber\\
&=& H_{s_2,\varphi_2}(\mathbf{h},\varphi_0)|_{U}.\nonumber
\end{eqnarray}
\end{proof}

From the above proposition, we have a biholomorphism
\begin{eqnarray}
H_*(\mathbf{h},\varphi_0):X(\lambda)^* \to X(\lambda')^*\nonumber
\end{eqnarray}
for each $\mathbf{h}\in {\rm Isom}(\lambda,\lambda')$ and $\varphi_0\in \mathcal{A}(k_{\mathbf{o}_{\lambda'},\mathbf{h}(\mathbf{o}_{\lambda})})$ by gluing $H_{s,\varphi}(\mathbf{h},\varphi_0)$.

Since the submanifold $X(\lambda)\backslash X(\lambda)^*$ is codimension $2$ in $X(\lambda)$, then the above map $H_*(\mathbf{h},\varphi_0)$ is extended to $H(\mathbf{h},\varphi_0):X(\lambda) \to X(\lambda')$ by Hartogs' extension theorem and we have completed the proof of Theorem \ref{main theorem2}.

\section{Applications}
The Riemannian metric on $X(\lambda)$ induced from the hyperk\"ahler structure $\omega_{\lambda}$ becomes Ricci-flat since the holonomy group of hyper-K\"ahler metric is contained in $Sp(1)$. It is shown in \cite{G1} that the Riemannian metric is complete.

Put $\mathbf{I} = \mathbb{Z}_{>0}$ and define $\lambda^{(\beta)}\in ({\rm Im}\mathbb{H})_0^\mathbf{I}$ by
\begin{eqnarray}
\lambda^{(\beta)}_n := n^{\beta}i\nonumber
\end{eqnarray}
for $\beta >1$. Let $f_{\beta_1,\beta_2}(t):=t^{\frac{\beta_2}{\beta_1}}$ for $t\ge 0$ and $f_{\beta_1,\beta_2}(t):=t$ for $t<0$ and $\beta_1,\beta_2 > 1$. Then we have $\mathbf{h}_{\beta_1,\beta_2}\in {\rm Isom}(\lambda^{(\beta_1)},\lambda^{(\beta_2)})$ defined by
\begin{eqnarray}
\mathbf{h}_{\beta_1,\beta_2}(\pi_{\lambda^{(\beta_1)}}(t,z)) := \pi_{\lambda^{(\beta_2)}}(f_{\beta_1,\beta_2}(t),z).\nonumber
\end{eqnarray}
Therefore $(X(\lambda^{(\beta_1)}),\omega_{\lambda^{(\beta_1)},\mathbb{C}})$ is isomorphic to $(X(\lambda^{(\beta_2)}),\omega_{\lambda^{(\beta_2)},\mathbb{C}})$ as holomorphic symplectic manifolds from Section \ref{main result}.

Now we denote by $g_{\beta}$ the Ricci-flat K\"ahler metric induced from the hyper-K\"ahler structure $\omega_{\lambda^{(\beta)}}$.
According to \cite{Hat}, the volume $V_{g_{\beta}}(p_0,r)$ of the geodesic ball in $X(\lambda^{(\beta)})$ with respect to ${g_{\beta}}$ of radius $r>0$ centered at $p_0\in X(\lambda^{(\beta)})$ satisfies
\begin{eqnarray}
0 < \liminf_{r\to +\infty}\frac{V_{{g_{\beta}}}(p_0,r)}{r^{4-\frac{2}{\beta + 1}}} \le \limsup_{r\to +\infty}\frac{V_{{g_{\beta}}}(p_0,r)}{r^{4-\frac{2}{\beta + 1}}} < +\infty.\nonumber
\end{eqnarray}
Thus we have the following result by putting $\alpha = 4-\frac{2}{\beta + 1}$.
\begin{thm}
There exist a complex manifold of dimension $2$ who has a family of complete Ricci-flat K\"ahler metrics $\{ g_{\alpha}\}_{3<\alpha <4}$ with
\begin{eqnarray}
0 < \liminf_{r\to +\infty}\frac{V_{g_{\alpha}}(p_0,r)}{r^{\alpha}} \le \limsup_{r\to +\infty}\frac{V_{g_{\alpha}}(p_0,r)}{r^{\alpha}} < +\infty.\nonumber
\end{eqnarray}
\end{thm}

The above argument can be generalized as follows. 
Let $\mathbf{I} = \mathbb{Z}_{>0}$ and take $\lambda, \lambda'\in ({\rm Im}\mathbb{H})_0^\mathbf{I}$ to be 
\begin{eqnarray}
\lambda_n = a_n i,\ \lambda'_n = a_n' i,\nonumber
\end{eqnarray}
where $a_n,a_n'\in \mathbb{R}$ satisfy $a_1<a_2<\cdots$ and $a'_1<a'_2<\cdots$. Then there exists a homeomorphism $f : \mathbb{R} \to \mathbb{R}$ such that $f(a_n) = a'_n$, and we can construct $\mathbf{h}\in {\rm Isom}(\lambda,\lambda')$.

Moreover, we can consider more general settings.
Let $\triangle_{\lambda} := \{\lambda_{n,\mathbb{C}}\in \mathbb{C}; \ n\in\mathbf{I}\} $ and $\triangle_{\lambda'}$ are discrete and closed subsets of $\mathbb{C}$. Assume $\triangle_{\lambda} = \triangle_{\lambda'}$, and for each $z\in \triangle_{\lambda}$, $F(\lambda,z) := \{\lambda_{n,\mathbb{R}}\in\mathbb{R};\ \lambda_{n,\mathbb{C}} = z\}$ and $F(\lambda',z)$ are isomorphic as ordered sets. Here, the order structures on $F(\lambda,z)$ is naturally induced from $\mathbb{R}$.
Under these assumptions, we may construct a homeomorphism $f_z:\mathbb{R}\to\mathbb{R}$ such that $f_z(F(\lambda,z)) = F(\lambda',z)$ for each $z\in \triangle_{\lambda} = \triangle_{\lambda'}$, then extend them to a homeomorphism $\tilde{f} : {\rm Im}\mathbb{H} \to {\rm Im}\mathbb{H}$ such that $\tilde{f} (t,z) = (f_z(t),z)$ for $z\in \triangle_{\lambda} = \triangle_{\lambda'}$. Thus we have the following result.
\begin{thm}
Let $\lambda, \lambda'\in ({\rm Im}\mathbb{H})_0^\mathbf{I}$ be generic and satisfy $\triangle_{\lambda} = \triangle_{\lambda'}$. If $\triangle_{\lambda}\subset\mathbb{C}$ is discrete and closed and $F(\lambda,z) \cong F(\lambda',z)$ as ordered sets for each $z\in\triangle_{\lambda}$, then $X(\lambda) \cong X(\lambda')$ as holomorphic symplectic manifolds.\label{thm 6.2}
\end{thm}

{\bf Acknowledgment.} The author was supported by Global COE program ``The research and training center for new development in Mathematics" of Graduate School of Mathematical Sciences, the University of Tokyo.

\bibliographystyle{plain}

\end{document}